\documentclass[12pt,a4paper]{amsart}
\usepackage{amssymb,amsmath}

\headsep=1cm \numberwithin{equation}{section}
\newtheorem{theorem}{Theorem}[section]
\newtheorem{lemma}[theorem]{Lemma}
\newtheorem{proposition}[theorem]{Proposition}
\newtheorem{corollary}[theorem]{Corollary}

\theoremstyle{definition}
\newtheorem{definition}[theorem]{Definition}

\newtheorem{remark}[theorem]{Remark}

\newcommand\Supp{\operatorname{Supp}}
\newcommand\Ass{\operatorname{Ass}}
\newcommand\mAss{\operatorname{mAss}}

\newcommand\Ann{\operatorname{Ann}}
\newcommand\Spec{\operatorname{Spec}}
\newcommand\Rad{\operatorname{Rad}}

\newcommand\Ext{\operatorname{Ext}}

\begin{document}
\title[Quintasymptotic primes and ideal topologies]{Asymptotic behaviour of integral closures, quintasymptotic primes and ideal topologies}%
\author[Reza naghipour and Peter Schenzel]{Reza Naghipour$^*$ and Peter Schenzel}
\address{Department of mathematics, University of Tabriz, Tabriz, Iran;
and School of Mathematics, Institute for Research in Fundamental
Sciences (IPM), P.O. Box: 19395-5746, Tehran, Iran.}
\email{naghipour@ipm.ir}  \email{naghipour@tabrizu.ac.ir}%
\address{Martin-l\"uther-universit\"at halle-wittenberg, fachbereich mathematik and informatik, d-06099 halle (saale), germany}%
\email{schenzel@mathematik.uni-halle.de}%
\thanks{ 
The first author thanks the Martin-L\"uther-Universit\"at  Halle-Wittenberg for the hospitality and facilities offered  during the preparation of this paper.\\
$^*$Corresponding author: e-mail: {\it naghipour@ipm.ir} (Reza Naghipour)}%

\begin{abstract}
Let $R$ be a  Noetherian  ring,  $N$  a
finitely generated $R$-module and $I$ an ideal of $R$. It is  shown that the sequences
$\Ass _R R/(I^n)_a^{(N)}$, $\Ass _R (I^n)_a^{(N)}/ (I^{n+1})^{(N)}_a$ and   $\Ass _R (I^n)_a^{(N)}/ (I^n)_a, n= 1,2, \dots,$
of associated prime ideals, are increasing and ultimately constant
for large $n$.
Moreover,  it is  shown that, if $S$ is a  multiplicatively closed subset of $R$,
then the topologies defined by $(I^n)_a^{(N)}$ and
$S((I^n)_a^{(N)}), \,{n\geq1}$,  are equivalent if and only if $S$ is disjoint
from the quintasymptotic primes of $I$.
By using this, we also show that, if $(R, \mathfrak{m})$ is
local and $N$ is quasi-unmixed,
then the local cohomology module $H^{\dim N}_I(N)$ vanishes if and only
if there exists a multiplicatively closed subset $S$ of $R$ such
that $\mathfrak{m} \cap S \neq \emptyset$ and that the topologies
induced by $(I^n)_a^{(N)}$ and
$S((I^n)_a^{(N)}), \, {n\geq1},$ are equivalent.
\end{abstract}
\maketitle

\section {Introduction}
The important concept of integral closure of an
ideal of a commutative Noetherian ring (with identity),  developed
by D. G. Northcott and D. Rees in \cite{NR}, is fundamental to a considerable body of recent and current  research both in
commutative algebra and algebraic geometry. Let $R$ be a
commutative ring (with identity), $I$ an ideal of $R$.  In the case when $R$ is Noetherian, we denote by $(I)_a$ the integral closure
of $I$, i.e., $(I)_a$ is the ideal of $R$ consisting of all
elements $x\in R$ which satisfy  an equation $x^n+ r_1x^{n-1}+
\cdots + r_n= 0$, where $r_i\in I^i, i=1, \ldots, n$.

In \cite{R} L.J. Ratliff, Jr., has shown that (when $R$ is Noetherian),  the sequence of
associated prime ideals
\[
\Ass_R R/(I^n)_a,  n= 1,2, \ldots ,
\]
is increasing and ultimately constant;  we  use the notation $A^*_a(I)$ to
denote   $\Ass_R R/(I^n)_a$ for large $n$.

The notion  of  integral closures of
ideals of $R$   relative  to a  Noetherian $R$-module $N$,
 was initiated by R.Y. Sharp  et al., in \cite{STY}.
  An element $x\in R$ is said to
be {\it integrally dependent on $I$ relative  to} $N$ if there
exists a positive integer $n$ such that
$x^{n}N \subseteq \sum_{i=1}^n x^{n-i}I^iN.$
Then the  set
\begin{center}
$I^{(N)}_a=\{x\in R\,|\, x$ is integrally dependent on $I$ relative  to $N$\}
\end{center}
 is an ideal of $R$, called the {\it integral closure of $I$ relative  to}
$N$, in the case $N=R$, $I^{(N)}_a$ is the classical integral closure  $I_a$ of $I$.
 It is clear that $I\subseteq I^{(N)}_a.$ We say that $I$ is
{\it integrally closed} relative  to $N$ if $I= I^{(N)}_a.$

In the second section (among other things) we
show that, when $R$ is a Noetherian ring and $N$ is a  finitely generated $R$-module,  the sequences
\[
\Ass_R R/(I^n)^{(N)}_a, \,\,\,  \Ass _R (I^n)_a^{(N)}/ (I^{n+1})^{(N)}_a \text{ and }   \Ass _R (I^n)_a^{(N)}/ ((I+\Ann_R N)^n)_a, \,\, n= 1,2, \ldots,
\]
of associated  primes,  are ultimately constant; we let $A^*_a(I, N):=\Ass_R R/(I^n)^{(N)}_a$ and
 $C^*_a(I, N):=\Ass _R (I^n)_a^{(N)}/ ((I+\Ann_R N)^n)_a$, for large $n$. Pursing this point of view further we shall show that 
 $A^*_a(I+ \Ann_R N)\setminus C^*_a(I, N) \subseteq A^*_a(I, N).$

 In \cite{Mc2}, McAdam studied the following interesting
set of prime ideals of $R$ associated with $I$,
\[
\bar{Q}^*(I)= \{\mathfrak{p} \in \Spec R  : \text{ there  exists
a } \mathfrak{q} \in \mAss \hat{R}_{\mathfrak{p}} \text{ such
that} \Rad (I\hat{R}_{\mathfrak{p}}+ \mathfrak{q})=
\mathfrak{p}\hat{R}_{\mathfrak{p}}\},
\]
and he called  $\bar{Q}^*(I)$  the set of {\it quintasymptotic prime ideals} of $I$.

 On the other hand,  Ahn  in \cite{Ah} extended the notion  of the quintasymptotic prime ideals
to a finitely generated module over $R.$ More
precisely, if $N$ is a finitely generated $R$-module then a prime
ideal $\mathfrak{p}$ of $R$ is said to be a {\it quintasymptotic
prime ideal} of $I$ with respect to $N$ whenever there exists a
$\mathfrak{q}\in \mAss_{\hat{R}_\mathfrak{p}}\hat{N}_\mathfrak{p}$
such that $\Rad(I\hat{R}_{\mathfrak{p}}+ \mathfrak{q})=
\mathfrak{p}\hat{R}_{\mathfrak{p}}.$ The set of all {\it
quintasymptotic prime ideals} of $I$ with respect to $N$ is denoted by
$\bar{Q}^*(I,N).$

In the third section, for a multiplicatively closed subset $S$ of $R$, we examine the equivalence  between the topologies defined by the filtrations $\{(I^n)_a^{(N)}\}_{n\geq1}$,
$\{S((I^n)_a^{(N)})\}_{n\geq1}$, $\{S(((I+ \Ann_RN)^n)_a)\}_{n\geq 1}$  and $\{S((I+ \Ann_RN)^n)\}_{n\geq 1}$  by using
the quintasymptotic prime ideals of $I$ with respect to $N$. Some of these results has been established,  by Schenzel in \cite{Sc, Sc1},
McAdam in \cite{Mc2} and Mehrvarz  et al., in \cite{MNS}, in certain case when $N=R$.

 A typical result in this direction is the following:
\begin{theorem}
Let $N$ be  a finitely generated
module over a  Noetherian  ring  $R$ and let $I$ be an ideal of $R$.  Let $S$ be a multiplicatively closed subset  of $R$. Then
the topologies defined by  $(I^n)_a^{(N)}$,
$S((I^n)_a^{(N)})$, $S(((I+ \Ann_RN)^n)_a)$  and  $S((I+ \Ann_RN)^n), \, {n\geq 1},$  are
equivalent if and only if $S$ is disjoint from each of the quintasymptotic prime ideals of $I$ with respect to $N$.
\end{theorem}

The proof of Theorem 1.1 is given in Theorem 3.11. One of our tools for proving Theorem 1.1 is the following,
which is a characterization of the quintasymptotic prime ideals of $I$  with respect to $N$. In the following, we use $I_a^{\langle N \rangle}$ to denote the
union $I_a^{(N)}:_Rs$, where $s$ varies in $R\backslash \bigcup \{\frak p\in \mAss_RN/IN\}$;
in particular, for every integer $k\geq1$ and every prime ideal
$\mathfrak{p}$ of $R$,
\[
(\mathfrak{p}^k)^{\langle N \rangle}_a= \bigcup _{s\in R\backslash
\mathfrak{p}}((\mathfrak{p}^k)^{(N)}_a:_Rs).
\]

\begin{proposition} \label{3.1}
Let $R$ be a Noetherian ring and let $N$ be a  finitely
generated $R$-module. Let $I \subseteq \mathfrak{p}$
be ideals of $R$ such that $\mathfrak{p}\in \Supp(N).$ Then  $\mathfrak{p}\in \bar{Q}^*(I, N)$
if and only if  there exists an integer $k \geq 0$ such that for all
integers $m \geq 0$
\[
(I^m)^{(N)}_a :_R \langle \mathfrak{p}\rangle \nsubseteq
(\mathfrak{p}^k)^{\langle N \rangle}_a.
\]
\end{proposition}

 Finally in this section we derive the following consequence of Theorem 1.1.
\begin{corollary} \label{3.3}
Let $R$ be a Noetherian ring,  $N$  a  finitely
generated $R$-module and $I$ an ideal of $R.$ Then the following
conditions are equivalent:
\begin{itemize}
\item[(i)] $\bar{Q}^*(I, N)= \mAss_R N/IN.$

\item[(ii)] The topologies defined by $\{(I^n)^{(N)}_a \}_{n \geq
0}$ and $\{(I^n)^{\langle N \rangle}_a \}_{n \geq 0}$ are equivalent.
\end{itemize}
\end{corollary}

For any ideal $I$ of $R$ and any $R$-module $N$, the i-th {\it
local cohomology module of $N$ with respect to} $I$ is defined by
\[
H^i_I(N):= \varinjlim \Ext^i_R(R/I^n,N).
\]
We refer the reader to \cite{BS} for basic properties of local cohomology modules.
The purpose of the fourth section is to characterize  the equivalence between the topologies defined by $(I^n)_a^{(N)}$ and
$S((I^n)_a^{(N)}), \, {n\geq1}$  in terms of the top local cohomology module $H^{\dim N}_I(N)$. This will
generalize the main result of Marti-Farre \cite{MF}, as an extension of the main results of Call \cite[Corollary 1.4]{Ca},
 Call-Sharp \cite{CS} and Schenzel \cite[Corollary  4.3]{Sc2}.
\begin{theorem}
 If $(R, \mathfrak{m})$ is a  local (Noetherian)  ring and $N$ a finitely generated
quasi-unmixed $R$-module of dimension $d$, then $H^d_I(N)= 0$ if
and only if there exists a multiplicatively closed subset $S$ of
$R$ such that $\mathfrak{m} \cap S \neq \emptyset$ and the
topologies induced by $(I^n)_a^{(N)}$ and
$S((I^n)_a^{(N)}), \, {n\geq1}$ are equivalent.
\end{theorem}

The result in Theorem 1.4 is proved in Theorem  4.1.  Pursing this point of view further we show that the support of the
$(d-1)$-th local cohomology module of a finitely generated $R$-module $N$  is always
finite ($d= \dim N$), which will be the strengthened and a generalized
version of a corresponding one by  Marley (\cite[Corollaries 2.4 and 2.5]{M}) and by Naghipour-Sedghi (\cite[Corollary 3.3]{NS}).
\begin{theorem}
Assume that $R$ is a Noetherian ring. Let $N$ be a
finitely generated $R$-module of dimension $d$ and $I$
an ideal of $R.$ Then $\Supp(H^d_I(N)) \subseteq \bar{Q}^*(I, N).$ Moreover, if
$(R, \frak m)$ is   local,  then
 $$\Supp(H^{d-1}_I(N)) \subseteq \bar{Q}^*(I,N) \cup \{\mathfrak{m}\}.$$
\end{theorem}
The proof of Theorem 1.5 is given in Corollaries 4.2 and 4.3.\\

Throughout the paper, all rings are commutative, with
identity, unless otherwise specified. We shall use $R$ to denote
such a ring, $I$ an ideal of $R$, and $N$ a non-zero module over
$R.$ If $(R, \mathfrak{m})$ is a Noetherian local ring and $N$ a
finitely generated $R$-module, then $\hat{R}$ (resp. $\hat{N}$)
denotes the completion of $R$ (resp. $N$) with respect to the
$\mathfrak{m}$-adic topology.  Then $N$  is
said to be {\it quasi-unmixed } if for every $\mathfrak{p}\in
\mAss_{\hat{R}}\hat{N}$, the condition $\dim\hat{R}/\mathfrak{p}=
\dim N$ is satisfied.
For any ideal $J$ of $R$, the
radical of $J$, denoted by $\Rad (J)$, is defined to be the set
$\{ a\in R:\, a^n\in J \text{ for some }n \in \mathbb{N}\}.$
Moreover we use $V(J)$ to denote the set of prime ideals of $R$
containing $J.$ Finally, for any $R$-module $L$ we shall use
$\mAss_RL$ to denote the set of minimal elements of $\Ass_RL.$ For
any unexplained notation or terminology we refer the reader to
\cite{Ma} and \cite{Na}.

\section{Asymptotic behaviour of integral closures of ideals}
The purpose of  this section is to study the asymptotic behaviour  of the  integral closure of
ideals  with respect to a finitely generated module $N$ over a Noetherian ring $R$. More precisely we show that the sequences
\[
\{\Ass_R R/(I^n)^{(N)}_a\}_{n\geq1},\,\,   \{\Ass _R (I^n)_a^{(N)}/ (I^{n+1})^{(N)}_a\}_{n\geq1}, \,\,\,   \{\Ass _R (I^n)_a^{(N)}/ ((I+\Ann_RN)^n)_a\}_{n\geq1},
\]
of associated  prime ideals, are ultimately constant;  and pursing this point of view further we  show that 
 $A^*_a(I+ \Ann_R N)\setminus C^*_a(I, N) \subseteq A^*_a(I, N).$

\begin{lemma} \label{2.1}
Let $R$ be a ring (not necessarily Noetherian), and  $N$  a
Noetherian $R$-module. Then for any ideal $I$ of $R$, the
following statements hold:
\begin{itemize}
\item[(i)] $I_a + {\rm Ann}_RN \subseteq I^{(N)}_a.$

\item[(ii)] $0^{(N)}_a= \Rad( \Ann_RN).$

\item[(iii)] $I^{(N)}_a/\Ann_RN = (I+ \Ann_RN / \Ann_RN)_a;$ so
that $\Rad (I^{(N)}_a) = \Rad (I+ \Ann_RN).$

\item[(iv)] For any multiplicatively closed subset  $S$ of $R$,
 $(S^{-1}I)^{(S^{-1}N)}_a = S^{-1}(I^{(N)}_a).$

\item[(v)] $\bigcap_{n\geq1}(I^n)^{(N)}_a= \bigcap\{\mathfrak{p}\in
\mAss_RN\,|\, \mathfrak{p}+ I\neq R\}.$ In particular, if $R$ is local  then  $$\bigcap_{n\geq1}(I^n)^{(N)}_a=\Rad( \Ann_RN).$$

\item[(vi)]  $(I^{(N)}_a J^{(N)}_a)^{(N)}_a = (IJ)^{(N)}_a$, where $J$ is a second ideal of $R$.
\end{itemize}
\end{lemma}

\begin{proof} (i) and (ii) follow from the definition. For
the proof of (iii) see \cite[Remark 1.6]{STY}.  (iv) follows from (iii) and \cite[Lemma 2.3]{S}. To show (v) use (iii) and
\cite[Lemma 3.11]{Mc1}. Finally, in order to prove (vi), use (iii) and
the fact that $(KL)_a = (K_aL_a)_a$ for all ideals $K$ and $L$ of $R.$
\end{proof}

The following corollary extends McAdam's result \cite[Lemma 1.4]{Mc2}.
\begin{corollary} \label{2.6}
Let $R$  and  $N$ be as in Lemma {\rm2.1}.  Let $I$ be an integrally closed ideal
with respect to $N.$ Then  $I$ has a primary decomposition each primary component of which is integrally closed with respect to $N.$
\end{corollary}
\begin{proof}
The result follows from Lemma \ref{2.1} and \cite[Lemma  1.4]{Mc2}.
\end{proof}

\begin{lemma} \label{2.2} Let $R$ be a ring (not necessarily Noetherian), and  $N$  a
Noetherian $R$-module.  If $I$ is an ideal of $R$ such that  is not contained in any of the minimal prime ideals of $\Ann_RN$, then
\begin{align*}
((I^n)^{(N)}_a :_R I^m )&= ((I^n)^{(N)}_a :_R (I^m)_a) \\
                        &=((I^n)^{(N)}_a :_R (I^m)^{(N)}_a)\\
                        & = (I^{n-m})^{(N)}_a
\end{align*}
for all integers $n\,\geq m\,\geq \,0.$

\end{lemma}

\begin{proof} In view of Lemma \ref{2.1} it is enough to show that
\[
((I^n)^{(N)}_a :_R I^m) \subseteq (I^{n-m})^{(N)}_a.
\]
To this end, let $x\in R$ be such that $I^m x \subseteq
(I^n)^{(N)}_a.$ To simplify notation let us denote the
Noetherian ring $R/{\rm Ann}_RN$ by $\widetilde{R};$ the natural
image of $x$ in $\widetilde{R}$ is denoted by $\widetilde{x};$  and for
each ideal $J$ of $R$ write $\widetilde{J}$ for the ideal $J+
\Ann_RN/\Ann_RN.$ Then by Lemma \ref{2.1}
\[
\widetilde{x} \in ((\widetilde{I}^n)_a :
  _{\widetilde{R}} (\widetilde{I})^m).
\]
Now, it follows  from \cite[Lemma 11.27]{Mc1} that $x\in
(I^{n-m})^{(N)}_a,$ as desired.
\end{proof}

We are now ready to state and prove the main results of this
section. Namely, we prove that the sequences of associated primes
$\{\Ass_R R/(I^n)^{(N)}_a\}_{n\geq1}$,
$\{\Ass_R(I^n)^{(N)}_a/(I^{n+1})^{(N)}_a\}_{n\geq1}$  and  $\{\Ass_R(I^n)^{(N)}_a/((I+\Ann_RN)^n)_a\}_{n\geq 1}$
are increasing and become eventually constant.

\begin{theorem} \label{2.3}
Let $I$ denote an ideal of a ring $R$ (not necessarily Noetherian),   and let $N$ be a
Noetherian $R$-module. Then the sequence  of associated
primes
\[
\Ass_R R/(I^n)^{(N)}_a ,  n= 1,2, \ldots,
\]
is increasing and ultimately constant.
Moreover, if $I$   is not contained in any of the minimal prime ideals of $\Ann_RN$, then the sequence
\[
\Ass_R (I^n)^{(N)}_a/
(I^{n+1})^{(N)}_a,  n= 1,2, \ldots,
\]
is also  increasing and eventually constant.
\end{theorem}

\begin{proof} First assume that $n\geq 1$ is an integer and
$\mathfrak{p}\in \Spec R.$ In order to simplify notation, we will
use $\widetilde{R}$ to denote the commutative Noetherian ring
$R/\Ann_R N$  and for each ideal $J$ of $R$ we will write
$\widetilde{J}$ for the ideal $J+ \Ann_RN/\Ann_RN$ of
$\widetilde{R}.$ Then by Lemma \ref{2.1}, it is easy to see that
$\mathfrak{p}\in \Ass_R R/(I^n)^{(N)}_a$ if and only if
$\widetilde{\mathfrak{p}}\in \Ass_{\widetilde{R}} \, \widetilde{R}
/(\widetilde{I}^n)_a.$ Hence it follows that
\[
\bigcup _{n\geq 1}\Ass_R R/ (I^n)^{(N)}_a= \{\mathfrak{q}\cap R\,
|\, \mathfrak{q} \in A^*_a(\widetilde{I})\}.
\]
As, by Ratliff's Theorem \cite{R},  $A^*_a(\widetilde{I})$ is
finite it follows that the set $ \bigcup _{n\geq 1}\Ass_R R/
(I^n)^{(N)}_a$ is a finite set. Moreover, since the sequence
$\{\Ass_{\widetilde{R}} \, \widetilde{R} /(\widetilde{I}^n)_a\}_{n\geq 1}$
 is increasing, it turns out that the
sequence $\{\Ass_R R/(I^n)^{(N)}_a\}_{n\geq 1}$ is increasing, and
therefore ultimately constant.

In order to prove the second part, we
 assume that $I$ is not contained in any of the minimal prime ideals of $\Ann_RN$. Suppose that $\mathfrak{p}\in
\Ass_R R/(I^{n+1})^{(N)}_a.$ Then there exists an $x \in R
\setminus(I^{n+1})^{(N)}_a$ such that $\mathfrak{p}=
((I^{n+1})^{(N)}_a:_R x).$ Since $I \subseteq \mathfrak{p}$ it
follows that $Ix \subseteq (I^{n+1})^{(N)}_a.$ Hence by Lemma
\ref{2.2}, $x \in (I^n)^{(N)}_a.$ Whence we get $\mathfrak{p}\in
\Ass_R\,(I^n)^{(N)}_a/(I^{n+1})^{(N)}_a.$ Therefore, when $I$ is not contained in any of the minimal prime ideals of $\Ann_RN$, it follows that
\[
\Ass_R R/(I^{n+1})^{(N)}_a=
\Ass_R\,(I^n)^{(N)}_a/(I^{n+1})^{(N)}_a
\]
for all integers $n\geq 0.$ This finally completes the proof.
\end{proof}

\begin{theorem} \label{2.5}
Suppose that $R$ is a  Noetherian ring.  Let  $N$ denote a  finitely generated
$R$-module and let $I$ be an ideal of $R$  such that  is not contained in any of the minimal prime ideals of $\Ann_RN$. Then the sequence
\[
\{\Ass_R(I^n)^{(N)}_a/(I^n)_a\}_{n\geq 1}
\]
is increasing and eventually constant.
\end{theorem}

\begin{proof} Assume first that $\mathfrak{p}\in \Ass_R(I^n)^{(N)}_a/(I^n)_a$
for some integer $n \geq 1.$ Without loss of generality  we may
assume  that $(R, \mathfrak{p})$ is a local ring.  There exists an element $x\in
(I^n)^{(N)}_a$ such that $\mathfrak{p}= ((I^n)_a:_R x).$ Then,  in view of
Lemma \ref{2.2}, $((I^{n+1})_a:_R Ix)$ is a proper ideal of $R$ and
$\mathfrak{p}\subseteq(I^{n+1})_a:_R Ix.$ Thus $\mathfrak{p}= ((I^{n+1})_a:_R Ix).$ Because of $Ix
\subseteq (I^{n+1})^{(N)}_a$ it follows that $\mathfrak{p}\in
\Ass_R(I^{n+1})^{(N)}_a/(I^{n+1})_a.$ Therefore the sequence
$\{\Ass_R(I^n)^{(N)}_a/(I^n)_a\}_{n\geq 1}$ is increasing. As
\[
\bigcup_{n\geq1} \Ass_R(I^n)^{(N)}_a/(I^n)_a \subseteq
A^*_a(I),
\]
and $A^*_a(I)$ is finite, we deduce  that it is eventually
constant for large $n.$
\end{proof}

\begin{corollary}
Let $R$ be a  Noetherian ring,  $N$  a  finitely generated
$R$-module and $I$  an ideal of $R$  such that  is not contained in any of the minimal prime ideals of $\Ann_RN$. Then the sequence
 \[
\{\Ass_R (I^n)^{(N)}_a/((I+ \Ann_RN)^n)_a\}_{n\geq 1}
\]
is increasing and ultimately constant.
\end{corollary}

\begin{proof} This  follows from Theorem \ref{2.5},  when $I$ is replaced by $I+\Ann_RN$, since  the integral closure with respect to $N$
of their powers are the same.
\end{proof}
\begin{definition} \label{4.4}
Suppose that $R$ is a  Noetherian ring. Let $I$ be an
ideal of $R.$ Let $N$ denote a  finitely generated
$R$-module.
The  eventual constant values of the sequences
\[
\{\Ass_R R/(I^n)^{(N)}_a \}_{n\geq1} \text{ and }
\{\Ass_R(I^n)^{(N)}_a/((I+ \Ann_R N)^n)_a \}_{n\geq1}
\]
will be denoted by $A^*_a(I, N)$ and $C^*_a(I, N)$ respectively.

It is easy to see that $A^*_a(I, N)$ and $C^*_a(I, N)$ are stable
under localization. Moreover, $\mAss_RN/IN\subseteq A^*_a(I, N)$ and  $A^*_a(0, N)=\mAss_R N.$
\end{definition}

\begin{proposition} \label{4.5} Let $R$ be a  Noetherian ring,  $I$ an
ideal of $R$  and   $N$  a  finitely generated $R$-module. Then
 $A^*_a(I+ \Ann_R N)\setminus C^*_a(I, N) \subseteq A^*_a(I, N).$

\end{proposition}

\begin{proof} Let $\mathfrak{p}\in A^*_a(I+ \Ann_R N)\setminus C^*_a(I, N).$ Because $A^*_a(I+ \Ann_R N),$
$C^*_a(I, N)$ and $A^*_a(I, N)$ behave well under localization we
may assume that $(R, \mathfrak{p})$ is a local ring.  Let $\mathfrak{p} = ((I+ \Ann_R N)^n)_a:_R x)$ for some
$x\in R$ and for  large $n.$ Because $\mathfrak{p} \subseteq
((I^n)^{(N)}_a:_R x)$ and $\mathfrak{p} \not\in C^*_a(I, N)$ it
follows that $\mathfrak{p}=( (I^n)^{(N)}_a:_R x).$ Hence
$\mathfrak{p} \in A^*_a(I, N)$, as required.
\end{proof}

\begin{remark} \label{4.6}
Let $R$ be a  Noetherian ring,  $N$  a  finitely generated $R$-module and $I$  an ideal of $R$  such that  is not contained in any of the minimal prime ideals of $\Ann_RN$ and  
 $(I^nR_\mathfrak{p})^{(N_\mathfrak{p})}_a = ((IR_\mathfrak{p}+\Ann_{R_\mathfrak{p}}N_\mathfrak{p})^n)_a$,
  for large $n$ and for all $\mathfrak{p} \in C^*_a(I, N).$ Then $(I^n)^{(N)}_a = ((I+\Ann_RN)^n)_a$  for large $n$. For this,
let $C^*_a(I, N)= \{\mathfrak{p}_1, \ldots, \mathfrak{p}_t\}.$  Choose  an integer $k$ such that
\[
C^*_a(I, N)= \Ass_R (I^n)^{(N)}_a /((I+ \Ann_RN)^n)_a
\]
for all $n \geq k$, and let
$(I^nR_{\mathfrak{p}_i})^{(N_{\mathfrak{p}_i})}_a=
((IR_{\mathfrak{p}_i}+ \Ann_{R_\mathfrak{p_i}}
N_\mathfrak{p_i})^n)_a$ for all $i= 1, \ldots, t.$ Now, if
 $(I^n)^{(N)}_a\neq ((I+ \Ann_RN)^n)_a,$ then
there exists   $x\in (I^n)^{(N)}_a\setminus ((I+ \Ann_RN)^n)_a,$   and  so $((I+ \Ann_RN)^n)_a:_R x$ is a proper ideal of $R$.  As $R$ is Noetherian,   there
is   $r \in R$ such that $\mathfrak{p}:=((I+ \Ann_RN)^n)_a:_R rx$
is a prime ideal of $R,$  and  hence $\mathfrak{p}\in
C^*_a(I, N).$ Thus $(I^nR_\mathfrak{p})^{(N_\mathfrak{p})}_a=
((IR_\mathfrak{p}+ \Ann_{R_\mathfrak{p}}N_\mathfrak{p})^n)_a$;
so that $x/1 \in ((I+ \Ann_RN)^n)_aR_\mathfrak{p}.$ Hence there is $s
\in R \setminus \mathfrak{p}$ such that $sx \in ((I+
\Ann_RN)^n)_a.$ That is $s \in ((I+ \Ann_RN)^n)_a:_R x \subseteq
\mathfrak{p}$, which is a contradiction.

\end{remark}


\section{Quintasymptotic primes and ideal topologies}

In this section we study the equivalence of the topologies defined by  $(I^n)_a^{(N)}$,
$S((I^n)_a^{(N)})$, $S(((I+ \Ann_RN)^n)_a)$  and $S((I+ \Ann_RN)^n), \, {n\geq 1},$  by using
the quintasymptotic prime ideals of $I$ with respect to $N$. The main results are Proposition 3.10 and Theorem 3.11. As a consequence we show that
$\bar{Q}^*(I, N)= \mAss_R N/IN$ if and only if the topologies  $(I^n)^{(N)}_a$
 and $(I^n)^{\langle N \rangle}_a, \, {n \geq 1},$ are equivalent.  We begin with the following elementary result.

\begin{lemma} \label{2.7}
Let $R$ be a Noetherian ring and $N$ a  finitely generated
$R$-module. Let $T$ be a faithfully flat Noetherian ring extension
of $R.$ Then, for any ideal $I$ of $R,$
\[
(IT)^{(N\otimes_RT)}_a \cap R= I^{(N)}_a.\]
\end{lemma}

\begin{proof} Let $x\in (IT)^{(N\otimes_RT)}_a\cap R.$ Then, in view of
\cite[Corollary 1.5]{STY}, there is an integer $n\geq1$ such that
$$(IT+Tx)^{n+1}(N\otimes_RT)= IT(IT+Tx)^n(N\otimes_RT).$$ Hence
\[
(I+Rx)^{n+1}(N\otimes_RT)= I(I+Rx)^n(N\otimes_RT).
\]
Therefore $$(I+Rx)^{n+1}N \otimes_RT= I(I+Rx)^nN \otimes_RT.$$ Now, by
faithfully flatness, we deduce that $(I+Rx)^{n+1}N= I(I+Rx)^nN$; hence  $x\in I^{(N)}_a$, by \cite[Corollary 1.5]{STY}.
Therefore the conclusion follows, since the opposite inclusion is clear by the faithfully flatness of $T$ over $R$.

\end{proof}

\begin{remark} \label{2.8}
Before continuing, let us fix two notations, which are employed by
Schenzel in \cite{Sc3} and McAdam in \cite{Mc2} in the
case $N= R$, respectively.

For any multiplicatively closed subset $S$ of $R$ and for each
ideal $J$ of $R$ we use $S(J)$ to denote the ideal $\cup_{s\in
S}(J:_Rs).$ Note that
\[
\Ass_RR/S(J)= \{ \mathfrak{p}\in \Ass_R R/J \,:\,\mathfrak{p}\cap
S=\emptyset\}.
\]
In the case $N$ is a finitely generated $R$-module and $S=R\backslash \bigcup \{\frak p\in \mAss_RN/JN\}$, we use $J_a^{\langle N \rangle}$ to denote the ideal $S(J_a^{(N)})$.
In particular, for every integer $k\geq1$ and every prime ideal
$\mathfrak{p}$ of $R$,  we have
\[
(\mathfrak{p}^k)^{\langle N \rangle}_a= \bigcup _{s\in R\backslash
\mathfrak{p}}((\mathfrak{p}^k)^{(N)}_a:_Rs).
\]
\end{remark}

\begin{proposition} \label{2.9}
Let $R$ be a Noetherian ring and let $N$ be a finitely generated $R$-module.
\begin{itemize}
\item[(i)] If  $(R, \mathfrak{m})$ is  local and $\mathfrak{p}\in \mAss_RN,$ then there exists an  element
$x\in R$ not in $\mathfrak{p}$ such that  for every ideal $J$ of $R$
with  $\mathfrak{m}$ is minimal over  $J+ \mathfrak{p}$,  $x\in
J+ \Ann_RN$ or $\mathfrak{m}\in \Ass_RR/J+ \Ann_RN$.

\item[(ii)] If $\mathfrak{p}\in \Spec R$ and $\mathfrak{q}\in
\mAss_RN$  with $\mathfrak{q}\subseteq\mathfrak{p},$ then
there is an integer $k\geq1$ such that $\mathfrak{p}\in
\Ass_RR/J+ \Ann_RN$ for any ideal $J$ of $R$ with $J\subseteq
(\mathfrak{p}^k)^{\langle N \rangle}_a$ and $\mathfrak{p}\in \mAss_RR/J+
\mathfrak{q}.$
\end{itemize}
\end{proposition}

\begin{proof} In order to show (i),   let  $$\frak q_1\cap \dots \cap \frak q_t= \Ann_RN,$$ be an irredundant primary
decomposition of the ideal $\Ann_RN$ with $\frak q_1$ \,  $\frak p$-primary. (Note that $\mathfrak{p}\in \Ass_RR/\Ann_RN$.)
It follows from $\mathfrak{p}\in \mAss_RN$  that  $\bigcap _{i=2}^t\frak q_i\nsubseteq \frak p$. Hence there exists
$x\in \bigcap _{i=2}^t\frak q_i$ such that $x\not\in \frak p$.  Now, let $J$ be any ideal of $R$ such that ${\rm Rad}(J+ \mathfrak{p})= \mathfrak{m}$
and that $\mathfrak{m}\not\in \Ass_RR/J+ \Ann_RN$. It is enough for us to show that $x\in J+ \Ann_RN$. To this end, let
$$Q_1\cap \dots \cap Q_l=J+\Ann_RN,$$
 be an irredundant primary decomposition of the ideal $J+\Ann_RN$, with $Q_i$ is $\frak p_i$-primary ideal, for
all $i=1, \dots, l$. Then $\frak m\neq \frak p_i$, for
all $i=1, \dots, l$, and so  it follows from  ${\rm Rad}(J+ \mathfrak{p})= \mathfrak{m}$ that $\frak p\nsubseteq \frak p_i$.  Hence $\frak p^v\nsubseteq \frak p_i$,  where $v\geq1$ is an integer such that $\frak p^v\subseteq \frak q_1$. Therefore, because of
$x\frak q_1\subseteq \Ann_RN$, it follows that $x\frak p^v\subseteq  Q_i$, for all $i=1, \dots, l$. Consequently $x\in Q_1\cap \dots \cap Q_l$, so that
$x\in J+\Ann_RN,$ as required.

 In order to prove (ii), without loss of generality,  we may assume that $(R, \mathfrak{p})$ is local. Then
$(\mathfrak{p}^k)^{\langle N \rangle}_a= (\mathfrak{p}^k)^{(N)}_a.$ Now, let $x$ be
as in (i). Then, in view of Lemma \ref{2.1}(v),  there exists an
integer $k\geq1$ such that $x \not\in (\mathfrak{p}^k)^{(N)}_a.$ Therefore if $J$ is an ideal of $R$ such that  $J\subseteq
(\mathfrak{p}^k)^{\langle N \rangle}_a$ and $\mathfrak{p}\in \mAss_RR/J+
\mathfrak{q}$, then  $x\not\in J+\Ann_RN,$ and so it follows from (i) that  $\mathfrak{p}\in \Ass_RR/J+ \Ann_RN$.
\end{proof}

\begin{proposition} \label{2.10}
Let $I$ be an ideal of a Noetherian ring  $R$ and $S$ a multiplicatively closed subset  of $R.$
 Then for any finitely generated $R$-module $N$,
\[
\bigcap_{n\geq1}S((I^n)^{(N)}_a)= \bigcap \{\mathfrak{p}\in
\mAss_RN\,\mid\,(I+ \mathfrak{p})\cap S= \emptyset \}.\]
\end{proposition}

\begin{proof} Let $x\in \bigcap_{n\geq1}S((I^n)^{(N)}_a).$ Then, for all $n\geq1$,  there
exists $s\in S$ such that $sx \in (I^n)^{(N)}_a.$ Now let
$\mathfrak{p}\in \mAss_R N$ be such that $(\mathfrak{p}+ I)\cap S=
\emptyset.$ Then, it follows from Lemma \ref{2.1}(v) that  $x\in \mathfrak{p}.$

Conversely, suppose that $x\in \mathfrak{p}$ for all
$\mathfrak{p}\in \mAss_RN$ with $(\mathfrak{p}+ I)\cap S=
\emptyset.$ Then, by virtue of Lemma \ref{2.1}(v), $x/1\in
(S^{-1}I^n)^{(S^{-1}N)}_a$ for all $n\geq1.$ Hence, in view of  Lemma \ref{2.1}(iv),   $x/1\in
S^{-1}((I^n)^{(N)}_a)$, and so $sx \in (I^n)^{(N)}_a$ for some
$s\in S.$ Consequently, we have $x\in S((I^n)^{(N)}_a)$, as
required.
\end{proof}

\begin{theorem} \label{2.11}
Let $R$ be a Noetherian ring and let $N$ be a finitely
generated $R$-module. Let $I$ and $J$ be ideals of $R.$ Then
\[
\bigcap_{n\geq1}((I^n)^{(N)}_a:_R \langle J\rangle)=
\bigcap\{\mathfrak{p}\in \mAss_RN\,\mid\,J\nsubseteq {\rm Rad}(I+
\mathfrak{p})\}.
\]
\end{theorem}

\begin{proof} In view of Theorem \ref{2.3}  the set  $A^*_a(I,N):= \bigcup _{n\geq1}
\Ass_RR/(I^n)^{(N)}_a$ is  finite. Let $A^*_a(I,N)=
\{\mathfrak{p}_1,\ldots, \mathfrak{p}_t\}.$ Let $r$ be an integer
such that $0\leq r\leq t$ and $J\nsubseteq\bigcup_{i=1}^r
\mathfrak{p}_i$ but $J\subseteq \bigcap_{i=r+1}^t \mathfrak{p}_j.$
Then there exists an element $s\in J$ such that
$s\not\in\bigcup_{i=1}^r \mathfrak{p}_i.$ Suppose $S= \{ s^{i}
\,\mid\, i\geq0 \}.$ Then it easily seen that
$((I^n)^{(N)}_a:_R \langle J\rangle)= S((I^n)^{(N)}_a)$ for each
integer $n\geq1.$ Now in view of  Proposition \ref{2.10} it is
enough to show that $J\subseteq {\rm Rad}(I+ \mathfrak{p})$ if and
only if $s\in {\rm Rad}(I+ \mathfrak{p})$ for each
$\mathfrak{p}\in \mAss_R N.$ To do this, as $s\in J$ one direction is clear. For other direction, let $\mathfrak{q}$ be a
 minimal prime ideal over $I+ \mathfrak{p}.$ Then, as $s\in {\rm Rad}(I+ \mathfrak{p})$ and $I+ \mathfrak{p}\subseteq \frak q$, we have
$s\in\mathfrak{q}$, and hence in view of the choice of $s$, it suffices to show that
$\mathfrak{q}\in A^*_a(I,N).$ By virtue of Lemma \ref{2.1}, we may
assume that $R$ is local with maximal ideal $\mathfrak{q}.$ Let
$x$ be as in the Proposition  \ref{2.9}. Then by Lemma \ref{2.1}, there
is an integer $n\geq1$ such that $x\not\in (I^n)^{(N)}_a.$ Now it
is easy to see that $\mathfrak{q}$ is  minimal over $
(I^n)^{(N)}_a+\mathfrak{p}.$ Therefore,  it
follows from  Proposition  \ref{2.9} that $\mathfrak{q}\in \Ass_R R/(I^n)^{(N)}_a,$  and so
$\mathfrak{q}\in A^*_a(I,N)$, as required.\\
\end{proof}

\begin{corollary}
Let $R$ be a Noetherian ring and $I$ an ideal of $R$. Let  $N$ be a finitely
generated $R$-module and $\mathfrak{p}\in \mAss_RN$.  Then  $\mAss_RR/(I+\frak p) \subseteq A^*_a(I,N).$
\end{corollary}
\begin{proof}
The assertion follows from the last argument in the proof of Theorem \ref{2.11}.
\end{proof}

\begin{corollary}\label{3.7}
Let $(R,\mathfrak{m})$  be a  local (Noetherian)  ring and let $N$ be a finitely
generated $R$-module. Then, for any proper ideal $I$ of $R$,

\[
\bigcap_{n\geq1}((I^n)^{(N)}_a:_R \langle \mathfrak{m}\rangle)=
\bigcap\{\mathfrak{p}\in \mAss_RN\,\mid\,{\rm Rad}(I+
\mathfrak{p})\subsetneqq \mathfrak{m}\}.
\]
\end{corollary}
\begin{proof}
 The assertion follows from Theorem \ref{2.11}.\\
\end{proof}
\begin{proposition} \label{2.13}
Let $(R,\mathfrak{m})$ be a local (Noetherian) ring, $I$ a proper ideal of
$R$ and $N$ a  finitely generated $R$-module. Then the
following conditions are equivalent:
\begin{itemize}
\item[(i)] For all $\mathfrak{p}\in \mAss_RN,
\Rad(I+\mathfrak{p})\neq\mathfrak{m}.$

\item[(ii)] $\bigcap_{n\geq1}(((I+\Ann_RN)^n)_a:_R\langle
\mathfrak{m}\rangle)\subseteq {\rm Rad}(\Ann_RN).$

\item[(iii)] $\bigcap_{n\geq1}((I+\Ann_RN)^n:_R\langle
\mathfrak{m}\rangle)\subseteq {\rm Rad}(\Ann_RN).$

\item[(iv)] $\bigcap_{n\geq1}((I^n)^{(N)}_a:_R \langle
\mathfrak{m}\rangle)= \Rad(\Ann_RN).$
\end{itemize}
\end{proposition}

\begin{proof} (i) $\Longrightarrow$ (ii): In view of  Corollary \ref{3.7},
\[
\bigcap_{n\geq1}((I^n)^{(N)}_a:_R \langle\mathfrak{m}\rangle)= \Rad(\Ann_RN).
\]
Hence as
\[
\bigcap_{n\geq1}(((I+\Ann_RN)^n)_a:_R\langle \mathfrak{m}\rangle)
\subseteq \bigcap_{n\geq1}((I^n)^{(N)}_a:_R \langle
\mathfrak{m}\rangle),
\]
 it follows that (ii) holds.  The implication (ii) $\Longrightarrow$ (iii) is trivial.

In order to show that (iii) $\Longrightarrow$ (iv),  suppose the contrary, that is (iv) is not
true. Then $$\Rad(\Ann_RN) \subsetneqq
\bigcap_{n\geq1}((I^n)^{(N)}_a:_R \langle \mathfrak{m}\rangle).$$
Hence, according to Corollary 3.7, there exists $\mathfrak{p}\in
\mAss_RN$ such that $\Rad(I+ \mathfrak{p})= \mathfrak{m}.$ Moreover, applying
the assumption it is easily seen that $\Rad(I+ \Ann_RN) \neq
\mathfrak{m}.$  Therefore  $$\Rad((I+ \Ann_RN)^n:_R\langle \mathfrak{m} \rangle + \mathfrak{p})=
\mathfrak{m},$$  for each integer $n\geq1$.

 Now, let $x$ be as in the Proposition \ref{2.9}. Since
$\mathfrak{m}\not\in \Ass_R R/((I+\Ann_RN)^n:_R\langle \mathfrak{m}
\rangle)$ it follows that
$$x \in \bigcap_{n\geq1}((I+\Ann_RN)^n:_R\langle \mathfrak{m}\rangle).$$
 Thus
$x\in {\rm \Rad}(\Ann_RN)$, i.e.,  $x\in \mathfrak{p}$ which is a
contradiction.

 Finally, the implication (iv) $\Longrightarrow$ (i) follows  from Corollary 3.7.
\end{proof}

\begin{theorem} \label{2.14}
Let $(R,\mathfrak{m})$ be a  local (Noetherian)  ring, let $N$ be a finitely generated $R$-module, and let $I$ be an ideal of
$R.$ Then the following conditions are equivalent:
\begin{itemize}
\item[(i)] $\bigcap_{n\geq1}((I^n \hat{R})^{(\hat{N})}_a:_{\hat{R}}
\langle \mathfrak{m} \hat{R} \rangle)= \Rad (\Ann_ {\hat{R}}
\hat{N}).$

\item[(ii)] For all integers $n\geq1$ there exists an integer
$k\geq1$ such that
\[
(I+\Ann_RN)^k:_R\langle \mathfrak{m}\rangle \subseteq
(\mathfrak{m}^n)^{(N)}_a.
\]

\item[(iii)] For all integers $n\geq1$ there exists an integer
$k\geq1$ such that
\[
((I+\Ann_RN)^k)_a:_R\langle \mathfrak{m}\rangle \subseteq
(\mathfrak{m}^n)^{(N)}_a.
\]

\item[(iv)] For all integers $n\geq1$ there exists an integer
$k\geq1$ such that
\[
(I^k)^{(N)}_a:_R\langle \mathfrak{m} \rangle \subseteq
(\mathfrak{m}^n)^{(N)}_a.
\]
\end{itemize}
\end{theorem}

\begin{proof} Without loss of
generality we may assume that $(R,\mathfrak{m})$ is a complete
local ring as follows by virtue of the faithfully flatness of
$\hat{R}$. Now suppose that (i) is satisfied, then
\[
\bigcap_{n\geq1}((I^n)^{(N)}_a:_R \langle \mathfrak{m} \rangle/
\Rad(\Ann_AN))= 0.
\]
As $R/\Ann_RN$ is a complete local ring, Chevalley's Theorem (see
\cite[Theorem 30.1]{Na}) says that for all $n\geq1$ there exists an
integer $k\geq1$ such that
\[
((I^k)^{(N)}_a:_R \langle \mathfrak{m}
\rangle)/\Rad(\Ann_AN)\subseteq (\mathfrak{m}/\Rad(\Ann_RN))^n.
\]
Therefore
\[
((I^k)^{(N)}_a:_R \langle \mathfrak{m}\rangle) \subseteq
\mathfrak{m}^n+ \Rad(\Ann_RN) \subseteq
(\mathfrak{m}^n)^{(N)}_a,
\]
and so the statement (iv) is shown to be true.

The conclusions (iv) $\Longrightarrow$ (iii) and (iii) $ \Longrightarrow$
(ii) are obviously true. So, in order to complete the proof, it is
enough to show that (ii) $\Longrightarrow$ (i). To this end, suppose
that for all $n\geq1$ there exists an integer $k\geq1$ such that
$$(I+\Ann_RN)^k:_R\langle \mathfrak{m}\rangle \subseteq
(\mathfrak{m}^n)^{(N)}_a.$$ Then  in view of  Lemma \ref{2.1}  we have
$$\bigcap_{k\geq1}((I+\Ann_RN)^k:_R \langle \mathfrak{m}\rangle)
\subseteq \Rad(\Ann_RN).$$  Now use Proposition 3.8 to
complete the proof.
\end{proof}

We are now ready to  prove the first main result of this section. In fact there  is  a  characterization of
the quintasymptotic prime ideals of $I$ with respect to $N$, which is a generalization of \cite[Proposition 3.5]{Mc2}.
\begin{proposition} \label{3.1}
Let $R$ be a Noetherian ring and let $N$ be a  finitely
generated $R$-module. Let $I \subseteq \mathfrak{p}$
be ideals of $R$ such that $\mathfrak{p}\in \Supp(N).$ Then the
following conditions are equivalent:
\begin{itemize}
\item[(i)] $\mathfrak{p}\in \bar{Q}^*(I, N).$

\item[(ii)] There exists an integer $k\geq0$ such that
$\mathfrak{p}\in \Ass_R R/J+\Ann_RN$ for any ideal $J$ of $R$ with
$I\subseteq \Rad(J)$ and $J\subseteq (\mathfrak{p}^k)^{\langle N \rangle}_a.$

\item[(iii)]  There exists an integer $k \geq 0$ such that for all
integers $m \geq 0$
\[
(I+ \Ann_RN)^m :_R \langle \mathfrak{p}\rangle \nsubseteq
(\mathfrak{p}^k)^{\langle N \rangle}_a.
\]

\item[(iv)]  There exists an integer $k \geq 0$ such that for all
integers $m \geq 0$
\[
((I+ \Ann_RN)^m)_a :_R \langle \mathfrak{p}\rangle \nsubseteq
(\mathfrak{p}^k)^{\langle N \rangle}_a.
\]

\item[(v)]  There exists an integer $k \geq 0$ such that for all
integers $m \geq 0$
\[
(I^m)^{(N)}_a :_R \langle \mathfrak{p}\rangle \nsubseteq
(\mathfrak{p}^k)^{\langle N \rangle}_a.
\]
\end{itemize}
\end{proposition}

\begin{proof} (i) $\Longrightarrow$ (ii): Let $\mathfrak{p}\in
\bar{Q}^*(I, N).$ Then there exists a prime ideal $\mathfrak{q}\in
\mAss_{\hat{R_{\mathfrak{p}}}}\hat{N_{\mathfrak{p}}}$ such that
$\Rad(I\hat{R_{\mathfrak{p}}}+ \mathfrak{q})=
\mathfrak{p}\hat{R_{\mathfrak{p}}}.$ Now, let $k$ be as in the
Proposition \ref{2.9}(ii), applied to $\mathfrak{q}\in
\mAss_{\hat{R_{\mathfrak{p}}}}\hat{N_{\mathfrak{p}}}.$ Let $J$ be
any ideal of $R$ such that $I\subseteq \Rad(J)$ and $J\subseteq
(\mathfrak{p}^k)^{\langle N \rangle}_a.$ Then $I\hat{R_{\mathfrak{p}}} \subseteq
\Rad(J\hat{R_{\mathfrak{p}}})$ and
$J\hat{R_{\mathfrak{p}}}\subseteq (\mathfrak{p}^k
\hat{R_{\mathfrak{p}}})^{(\hat{N_{\mathfrak{p}}})}_a$ by virtue of
Lemma \ref{2.7}. Since $\mathfrak{p}\hat{R_{\mathfrak{p}}}$ is the
maximal ideal of $\hat{R_{\mathfrak{p}}}$, it follows that
$(\mathfrak{p}\hat{R_{\mathfrak{p}}})^{(\hat{N_{\mathfrak{p}}})}_a=
\mathfrak{p}\hat{R_{\mathfrak{p}}}$, and so
$J\hat{R_{\mathfrak{p}}}$ is a proper ideal of
$\hat{R_{\mathfrak{p}}}.$ Thus $\Rad(J\hat{R_{\mathfrak{p}}}+
\mathfrak{q})= \mathfrak{p}\hat{R_{\mathfrak{p}}}$.  Hence  Proposition
\ref{2.9} shows that
$\mathfrak{p}\hat{R_{\mathfrak{p}}}\in \Ass_{\hat {R_{\mathfrak{p}}}} \hat {R_{\mathfrak{p}}}/J\hat {R_{\mathfrak{p}}}+\Ann_{\hat {R_{\mathfrak{p}}}}\hat {N_{\mathfrak{p}}}$,
and so by \cite[Theorem 23.2]{Ma} we have $\mathfrak{p}\in \Ass_R R/J+\Ann_RN$. That is (ii) holds.

The implication (ii) $\Longrightarrow$ (iii) follows easily from the fact that  $$\mathfrak{p}\not\in \Ass_R R/((I+ \Ann_RN)^n :_R\langle
\mathfrak{p}\rangle),$$ for all integers $n \geq 0.$

The conclusions (iii) $\Longrightarrow$ (iv) and (iv) $\Longrightarrow$ (v) are obviously
true.

Finally, in order to complete the proof, we have to show the
implication (v) $\Longrightarrow$ (i). To this end, suppose that  there is an
integer $k \geq 0$ such that
$((I^m)^{(N)}_a :_R \langle \mathfrak{p}\rangle) \nsubseteq
(\mathfrak{p}^k)^{\langle N \rangle}_a$, for all integers $m \geq 0$.  Then by Lemma \ref{2.7},
$$((I^m\hat{R_{\mathfrak{p}}})^{(\hat{N_{\mathfrak{p}}})}_a :_
{\hat{R_{\mathfrak{p}}}} \langle \mathfrak{p}\hat{R_{\mathfrak{p}}}\rangle) \nsubseteq
(\mathfrak{p}^k\hat{R_{\mathfrak{p}}})^{(\hat{N_{\mathfrak{p}}})}_a.$$
Hence, in view of  Proposition \ref{2.13}, there is a
$\mathfrak{q}\in \mAss_{\hat{R_{\mathfrak{p}}}}
\hat{N_{\mathfrak{p}}}$ such that $\Rad(I\hat{R_{\mathfrak{p}}}+
\mathfrak{q})= \mathfrak{p}\hat{R_{\mathfrak{p}}}$, and so
$\mathfrak{p}\in \bar{Q}^*(I, N)$, as required.
\end{proof}

We are now ready to state and prove the second main theorem of this section, which is  a characterization of  the equivalence  between the topologies  $\{(I^n)_a^{(N)}\}_{n\geq1}$,
$\{S((I^n)_a^{(N)})\}_{n\geq1}$, $\{S(((I+ \Ann_RN)^n)_a)\}_{n\geq 1}$  and $\{S((I+ \Ann_RN)^n)\}_{n\geq 1}$   in terms of
the quintasymptotic primes of $I$ with respect to $N$.  This will  generalize  the main result of   McAdam \cite{Mc2}.
\begin{theorem}\label{3.2}
Let $R$ be a Noetherian ring,  $N$  a  finitely
generated $R$-module and $I$ an ideal of $R.$ Then for any
multiplicatively closed subset $S$ of $R$ the following
are equivalent:
\begin{itemize}
\item[(i)] $S \subseteq R\setminus\bigcup\{\mathfrak{p}\in
\bar{Q}^*(I, N)\}.$

\item[(ii)] The topologies defined by $\{S((I^n)^{(N)}_a)\}_{n
\geq 0}$ and $\{(I^n)^{(N)}_a \}_{n \geq 0}$ are equivalent.

\item[(iii)] The topology defined by $\{S(((I+ \Ann_RN)^n)_a)\}_{n
\geq 0}$ is finer than the topology defined by $\{(I^n)^{(N)}_a
\}_{n \geq 0}.$

\item[(iv)] The topology defined by $\{S((I+ \Ann_R N)^n) \}_{n
\geq 0}$ is finer than the topology defined by $\{(I^n)^{(N)}_a
\}_{n\geq0}.$

\item[(v)] For all $\mathfrak{p}\in \Supp(N) \cap V(I)$ the
topology defined by $\{S((I^n)^{(N)}_a)\}_{n\geq0}$ is finer than
the topology defined by $\{(\mathfrak{p}^n)^{\langle N \rangle}_a)\}_{n \geq
0}.$

\item[(vi)] For all $\mathfrak{p}\in \Supp(N) \cap V(I)$ the
topology defined by $\{S(((I+ \Ann_RN)^n)_a)\}_{n \geq 0}$
is finer than the topology defined by
$\{(\mathfrak{p}^n)^{\langle N \rangle}_a)\}_{n \geq 0}.$

\item[(vii)] For all $\mathfrak{p}\in \Supp(N) \cap V(I)$ the
topology defined by $\{S((I+\Ann_R N)^n) \}_{n \geq 0}$
is finer than the topology defined by
$\{(\mathfrak{p}^n)^{\langle N \rangle}_a)\}_{n \geq 0}.$
\end{itemize}
\end{theorem}

\begin{proof}  In order to show  (i)  $\Longrightarrow$ (ii),  let $\mathfrak{p}\in
\Spec R$ with $I+ \Ann_RN \subseteq \mathfrak{p}$ and let $l \geq
1.$ We first show that there exists an integer $m \geq 1$ such
that $S((I^m)^{(N)}_a) \subseteq (\mathfrak{p}^l)^{\langle N \rangle}_a.$ To do
this, let $S'$ be the natural image of $S$ in $R_\mathfrak{p}.$
Since $\bar{Q}^*(I, N)$ behaves well under localization, we have
$S' \subseteq R_\mathfrak{p}\setminus\cup\{\mathfrak{q}\in
\bar{Q}^*(IR_\mathfrak{p}, N_\mathfrak{p})\}.$ Moreover,  it is easy
to see that $S'((I^mR_\mathfrak{p})^{(N_\mathfrak{p})}_a)
\subseteq (\mathfrak{p}^lR_\mathfrak{p})^{\langle N_\mathfrak{p}\rangle }_a$
implies $S((I^n)^{(N)}_a) \subseteq (\mathfrak{p}^n)^{\langle N \rangle}_a.$
Therefore we may assume that $R$ is local with maximal ideal
$\mathfrak{p}.$ Then $(\mathfrak{p}^n)^{\langle N \rangle}_a=
(\mathfrak{p}^n)^{(N)}_a.$ By view of Lemma \ref{2.7} and
\cite[Proposition 3.8]{Ah} we may assume in addition that $R$ is complete. Whence, in view of \cite[Lemma 3.5]{Ah},
for any $\mathfrak{q} \in \mAss_R N$,   $S$ is
disjoint from $I+ \mathfrak{q}.$ Therefore, by Proposition \ref{2.10} we have $
\bigcap_{n \geq 1} S((I^n)^{(N)}_a) = \Rad (\Ann_RN).$ Consequently $$
\bigcap_{n\geq1}S((I^n)^{(N)}_a)/ \Rad(\Ann_R N)= 0.$$ As the ring
$R/\Rad(\Ann_RN)$ is complete, Chevalley's Theorem (see
\cite[Theorem 30.1]{Na}) implies the existence of an integer $m \geq 1$
such that $$S((I^m)^{(N)}_a)/\Rad(\Ann_RN) \subseteq
(\mathfrak{p}/\Rad(\Ann_RN))^l.$$  Hence
\[
S((I^m)^{(N)}_a) \subseteq \mathfrak{p}^l+ \Rad(\Ann_RN) \subseteq
(\mathfrak{p}^l)^{(N)}_a.
\]
Now, in view of Corollary \ref{2.6}, we can consider $\mathfrak{q}_1\cap \dots \cap \mathfrak{q}_n$ a
minimal primary decomposition of $(I^l)^{(N)}_a$ where
$\mathfrak{q}_i$ is $\mathfrak{p}_i$-primary and integrally closed
with respect to  $N$ for every $i= 1, \ldots, n.$ Then there exists an
integer $l_i$ such that $\mathfrak{p}_i^{l_i} \subseteq
\mathfrak{q}_i$ for $i= 1, \ldots, n$, and moreover for some $m_i$
we have $S((I^{m_i})^{(N)}_a)\subseteq
(\mathfrak{p}_i^{l_i})^{\langle N \rangle}_a.$  Let $m=
\max\{m_1,\ldots, m_n\}.$ Then we deduce that
$S((I^m)^{(N)}_a)\subseteq (\mathfrak{p}_i^{l_i})^{\langle N \rangle}_a$ for
each $1\leq i\leq n.$
On the other hand,  we have
$$(\mathfrak{p}_i^{l_i})^{\langle N \rangle}_a \subseteq \bigcup _{s\in
R\setminus \mathfrak{p}_i}((\mathfrak{q}_i)^{(N)}_a:_Rs)=
\mathfrak{q}_i,$$ and therefore $S((I^m)^{(N)}_a)\subseteq
\bigcap_{i=1}^n \mathfrak{q}_i.$ This completes the proof of (ii).

The implications (ii) $\Longrightarrow$ (iii) $\Longrightarrow$ (iv) are
obviously true. In order to prove the implication (iv)$\Longrightarrow$ (v),  it is enough to show that
\[
S \subseteq R\setminus \bigcup \{\mathfrak{p}\in \bar{Q}^*(I,
N)\}.
\]
To do this, let $\mathfrak{p}\in \bar{Q}^*(I, N).$ Then, by Proposition
\ref{3.1} there exists an integer $k \geq 0$ such that $((I+
\Ann_R N)^m :_R \langle \mathfrak{p}\rangle) \nsubseteq
(\mathfrak{p}^k)^{\langle N \rangle}_a$ for all integers $m \geq 0.$ On the
other hand, by the assumption there is an integer $l \geq 0$ such
that $S((I+ \Ann_RN)^l) \subseteq (I^k)^{(N)}_a.$ Therefore,
 $$(I+ \Ann_RN)^l :_R \langle \mathfrak{p}\rangle \nsubseteq
S((I+ \Ann_R N)^l).$$ Then it is readily to see that $\mathfrak{p}
\cap S = \emptyset$,  as required.

The conclusions (v) $\Longrightarrow$ (vi) $\Longrightarrow$ (vii) are
trivial. Finally, an argument similar to that used in the proof of
the implication (iv) $\Longrightarrow$ (v) shows that (vii)
$\Longrightarrow$ (ii) holds.
\end{proof}

An immediate consequence of the Theorem \ref{3.2} is the following corollary.

\begin{corollary} \label{3.3}
Let $R$ be a Noetherian ring,  $N$  a  finitely
generated $R$-module and $I$ an ideal of $R.$ Then the following
conditions are equivalent:
\begin{itemize}
\item[(i)] $\bar{Q}^*(I, N)= \mAss_R N/IN.$

\item[(ii)] The topologies defined by $\{(I^n)^{(N)}_a \}_{n \geq
0}$ and $\{(I^n)^{\langle N \rangle}_a \}_{n \geq 0}$ are equivalent.
\end{itemize}
\end{corollary}
\begin{proof}
Let $S=R\setminus\bigcup\{\mathfrak{p}\in \mAss_RN/IN\}.$ Then  $S((I^n)^{(N)}_a )=(I^n)^{\langle N\rangle}_a $. Now, if
$\bar{Q}^*(I, N)= \mAss_R N/IN,$  then   $S=R\setminus\bigcup\{\mathfrak{p}\in \bar{Q}^*(I, N)\}.$    Hence Theorem \ref{3.2} implies
that the topologies defined by $\{(I^n)^{(N)}_a \}_{n \geq 0}$ and $\{(I^n)^{\langle N \rangle}_a \}_{n \geq 0}$ are equivalent.  Conversely, if these
 topologies are equivalent, then it follows from Theorem \ref{3.2} that $S\subseteq R\setminus\bigcup\{\mathfrak{p}\in \bar{Q}^*(I, N)\}$, and so
 $\bar{Q}^*(I, N)\subseteq\mAss_R N/IN$. On the other hand,  using  \cite[Lemma 3.5]{Ah}, one easily  sees that  $\mAss_R N/IN\subseteq \bar{Q}^*(I, N)$;  and so
 $\bar{Q}^*(I, N)= \mAss_R N/IN.$ This completes the proof.
\end{proof}

\section{Local cohomology and ideal topologies}
The purpose of this section is to establish  the equivalence between the topologies defined by $\{(I^n)_a^{(N)}\}_{n\geq1}$ and
$\{S((I^n)_a^{(N)})\}_{n\geq1}$  in terms of the vanishing of the  top local cohomology module $H^{\dim N}_I(N)$. This will
generalize the main result of Marti-Farre \cite{MF}, as an extension of the main results of Call \cite[Corollary 1.4]{Ca},
 Call-Sharp \cite{CS} and Schenzel \cite[Corollary  4.3]{Sc2}.

\begin{theorem} \label{4.1}
Let $(R, \mathfrak{m})$ be a  local (Noetherian)  ring,  $N$
a  finitely generated $R$-module of dimension $d$ and
$I$ an ideal of $R.$  Consider the following conditions:
\begin{itemize}
\item[(i)] There exists a multiplicatively closed subset $S$ of
$R$ such that $ \mathfrak{m} \cap S \neq\emptyset$ and that the
topologies defined by $\{S((I^n)^{(N)}_a)\}_{n\geq0}$ and
$\{(I^n)^{(N)}_a \}_{n\geq0}$ are equivalent.

\item[(ii)] $H^d_I(N)= 0.$
\end{itemize}
Then \text{\rm(i)} $\Longrightarrow$ \text{\rm (ii)}; and these conditions are equivalent,
whenever $N$ is quasi-unmixed.
\end{theorem}

\begin{proof} We start with the proof of the implication (i)
$\Longrightarrow$ (ii). By Theorem \ref{3.2} we have $S
\subseteq R \setminus \bigcup \{\mathfrak{p}\in \bar{Q}^*(I, N)\}.$
Then  $\mathfrak{m} \not\in \bar{Q}^*(I, N).$
Therefore for all $\mathfrak{q}\in \mAss_{\hat{R}}\hat{N}$ we have
$\dim \hat{R}/I\hat{R}+ \mathfrak{q} >0.$ By the
Lichtenbaum-Hartshorne Theorem (see \cite[Corollary 3.4]{DS}), it follows
that $H^d_I(N)= 0.$

Now, assume that $N$ is quasi-unmixed and that (ii) holds. We show
that (i) is true. To this end, let $S= R \setminus \bigcup
\{\mathfrak{p} \in \bar{Q}^*(I, N)\}.$ Then, in view of Theorem
\ref{3.2},   the topologies defined by
$\{S((I^n)^{(N)}_a)\}_{n\geq0}$ and $\{(I^n)^{(N)}_a \}_{n\geq0}$
are equivalent. Hence, it is enough to show that $ \mathfrak{m}
\cap S \neq\emptyset.$ Suppose, the contrary, namely $
\mathfrak{m} \cap S = \emptyset.$ Then $\mathfrak{m}\in
\bar{Q}^*(I, N).$ So there exists $ \mathfrak{q}\in
\mAss_{\hat{R}}\hat{N}$ such that $\mathfrak{m}\hat{R}=
\Rad(I\hat{R}+ \mathfrak{q}).$ As $N$ is quasi-unmixed it follows
that $\dim \hat{R}/I\hat{R}+ \mathfrak{q}= 0$ for some $
\mathfrak{q} \in \mAss_{\hat{R}}\hat{N}$ such that $\dim
\hat{R}/\mathfrak{q}= d.$  Now,  use \cite[Corollary 3.4]{DS} to see that
$H^d_I(N) \neq 0$,  which is a contradiction.
\end{proof}

The final results will be the strengthened and a generalized
version of a corresponding one by  Marley (\cite[Corollaries 2.4 and  2.5]{M}) and  Naghipour-Sedghi  (\cite[Corollary 3.3]{NS}).
\begin{corollary} \label{3.5}
Assume that $R$ is a Noetherian ring. Let $N$ be a
finitely generated $R$-module of dimension $d$ and $I$
an ideal of $R.$ Then $\Supp(H^d_I(N)) \subseteq \bar{Q}^*(I, N).$ Moreover the equality holds, whenever $N$ is Cohen-Macaulay.
\end{corollary}

\begin{proof} Let $\mathfrak{p}\in \Supp(H^d_I(N)).$ Then
$H^d_{IR_\mathfrak{p}}(N_{\mathfrak{p}})\neq 0,$  and so $\dim
N_{\mathfrak{p}} = d.$ Hence, in view of  the
Lichtenbaum-Hartshorne Theorem (cf.  \cite[Corollary  3.4]{DS}) there
exists $ \mathfrak{q}\in \mAss_{\hat{R_{\mathfrak{p}}
}}\hat{N_{\mathfrak{p}}}$ such that
$\mathfrak{p}\hat{R_{\mathfrak{p}}}= \Rad(I\hat{R_{\mathfrak{p}}}+
\mathfrak{q}).$ Thus  $\mathfrak{p}\in
\bar{Q}^*(I, N)$, and  so $\Supp(H^d_I(N))\,\subseteq\,\bar{Q}^*(I,N).$

In order to prove the second assertion, let $\mathfrak{p}\in\bar{Q}^*(I, N)$. Then  there
exists $ \mathfrak{q}\in \mAss_{\hat{R_{\mathfrak{p}} }}\hat{N_{\mathfrak{p}}}$ such that
$\mathfrak{p}\hat{R_{\mathfrak{p}}}= \Rad(I\hat{R_{\mathfrak{p}}}+\mathfrak{q}).$  Now, since
$N$ is Cohen-Macaulay, we deduce that $\dim \hat{R_{\mathfrak{p}} }/\mathfrak q=\dim\hat{N_{\mathfrak{p}}}$.
Whence, in view of the Lichtenbaum-Hartshorne Theorem, $H^d_{IR_\mathfrak{p}}(N_{\mathfrak{p}})\neq 0,$  and so  $\mathfrak{p}\in \Supp(H^d_I(N)).$
\end{proof}

\begin{corollary} \label{3.6}
Let $(R, \mathfrak{m})$ be a local (Noetherian)  ring,  $N$
a  finitely generated $R$-module of dimension $d$ and
$I$ an ideal of $R.$  Then
\[
\Supp(H^{d-1}_I(N))\,\subseteq\,\bar{Q}^*(I, N) \cup\,
\{\mathfrak{m}\}.
\]
Therefore $\Ass_R H^{d-1}_I(N)$ is a finite set.
\end{corollary}

\begin{proof} Let $\mathfrak{p}\in \Supp(H^{d-1}_I(N))$ such that
$\mathfrak{p}\neq \mathfrak{m}.$ Then
$H^{d-1}_{IR_{\mathfrak{p}}}(N_{\mathfrak{p}})\neq 0$, and so
$\dim N_{\mathfrak{p}} = d-1.$  Now, by the proof of Corollary
\ref{3.5} we have  $\mathfrak{p}\in \bar{Q}^*(I, N)$, as required.
\end{proof}

\begin{center}
{\bf Acknowledgments}
\end{center}
The authors are deeply grateful to the
referee for providing numerous  suggestions to improve the readability of the paper and for
drawing the authors'  attention to Lemma 2.3 and Theorem  2.5.  Also, we would like to thank  Dr.  Monireh Sedghi for a  careful reading of  the original manuscript and helpful suggestions.


\end{document}